\documentclass{compositio}

\renewcommand{\contentsname}{Contents\\{\footnotesize\normalfont(A table
of contents should normally not be included)}}

\usepackage{amscd,amsmath,amssymb,amsfonts,bbm}
\usepackage[T1]{fontenc}
\usepackage{lmodern}

\usepackage{mathtools}
\usepackage{enumerate}
\usepackage{multicol}
\usepackage[all]{xy}

\usepackage[unicode,pdfborder={0 0 0},final]{hyperref}

\usepackage{footmisc}

\newtheorem{thm}{Theorem}[section]

\newtheorem{hyp}[thm]{Hypotheses}
\newtheorem{prop}[thm]{Proposition}
\newtheorem{lem}[thm]{Lemma}
\newtheorem{cor}[thm]{Corollary}

\theoremstyle{definition}

\theoremstyle{remark}
\newtheorem{rem}[thm]{Remark}

\numberwithin{equation}{section}

\DeclareMathOperator{\Q}{\mathbb{Q}}

\DeclareMathOperator{\sO}{\mathcal{O}}
\DeclareMathOperator{\sL}{\mathcal{L}}

\DeclareMathOperator{\Br}{Br}

\DeclareMathOperator{\Pic}{Pic}

\DeclareMathOperator{\Sing}{Sing}

\DeclareMathOperator{\Id}{Id}

\DeclareMathOperator{\R}{\mathbb{R}}
\DeclareMathOperator{\bP}{\mathbb{P}}

\DeclareMathOperator{\sm}{sm}

\DeclareMathOperator{\Gal}{Gal}
\DeclareMathOperator{\Z}{\mathbb{Z}}

\DeclareMathOperator{\Hdg}{Hdg}

\DeclareMathOperator{\ev}{ev}
\DeclareMathOperator{\rank}{rank}

\renewcommand{\C}{\mathbb{C}}

\begin{document}

\title{Sums of three squares and Noether-Lefschetz loci}
\author{Olivier Benoist}
\email{olivier.benoist@unistra.fr}
\address{Institut de Recherche Math\'ematique Avanc\'ee\\
UMR 7501, Universit\'e de Strasbourg et CNRS\\
7 rue Ren\'e Descartes\\
67000 Strasbourg, FRANCE.}

\classification{11E25 (primary), 14P99, 14D07, 14M12 (secondary).}
\keywords{Sums of squares, Hodge theory, real algebraic geometry, variations of Hodge structures.}

\begin{abstract}
We show that the set of real polynomials in two variables that are sums of three squares of rational functions is dense in the set of those that are positive semidefinite. We also prove that the set of real surfaces in $\bP^3$ whose function field has level $2$ is dense in the set of those that have no real points.
\end{abstract}

\maketitle

\section*{Introduction}\label{intro}

 \subsection{Sums of squares}
\label{sos}
  Let $\R[x_1,x_2]_{\leq d}$ be the space of real polynomials of degree $\leq d$ in two variables. Consider the cone $P_d\subset\R[x_1,x_2]_{\leq d}$ of polynomials that are positive semidefinite, i.e. that only take nonnegative values. Since an odd degree polynomial changes sign, we will assume that $d$ is even.

  It is known since Hilbert that every polynomial $f\in P_d$ is a sum of four squares in the field $\R(x_1,x_2)$ of rational functions (\cite{Hilbert}, see \cite[p.282]{Landau}).
 If $d\leq 4$, Hilbert \cite{Hilbert2} has shown the stronger statement that $f$ is a sum of three squares in $\R[x_1,x_2]$, but this does not extend in any way to degrees $d\geq 6$. Indeed, there exist polynomials $f\in P_6$ that are not sums of squares of polynomials (Hilbert \cite{Hilbert2}), and that are not sums of three squares of rational functions (Cassels-Ellison-Pfister \cite{CEP}). Motzkin's polynomial $1+x_1^2x_2^4+x_2^2x_1^4-3x_1^2x_2^2$ is an example of both phenomena.

  Sums of squares of polynomials form a closed cone $\Sigma_d\subset P_d$ \cite[Theorem 3]{Bergart}
hence are not dense in $P_d$ as soon as $d\geq 6$. At least when $d=6$, the structure of the cone $\Sigma_d$  is very well understood \cite{Blekh1, Blekh2}.

  Our goal is to complete the picture by studying the set $Q_d\subset P_d$ of polynomials that are sums of three squares in $\R(x_1,x_2)$. It is easily seen to be a countable union of closed subsets of $P_d$, indexed by the degrees of the denominators of the rational functions that appear in a representation of $f\in Q_d$ as a sum of three squares.
 Our main contribution is a proof of the density of this subset.

\begin{thm}\label{main}
The subset $Q_d\subset P_d$ is 
 a countable union of closed semialgebraic subsets of $P_d$, that has empty interior if $d\geq 6$. It is dense in $P_d$.
\end{thm}

  When $d\geq 6$, the set $P_d\setminus \Sigma_d$ of positive semidefinite polynomials that are not sums of squares of polynomials is a nonempty open subset of $P_d$.
  As a consequence of Theorem \ref{main}, $Q_d$ is dense in this open subset, showing the existence of many polynomials $f\in\R[x_1,x_2]_d$ that are sums of three squares in $\R(x_1,x_2)$ but not sums of squares of polynomials. The first examples of such polynomials had been constructed by Leep and Starr \cite[Theorem 2]{LeepStarr}.

\subsection{Strategy of the proof}  Our starting point is Colliot-Th\'el\`ene's Hodge-the\-ore\-tic proof of the Cassels-Ellison-Pfister theorem \cite{CTNL}: he associates to a polynomial $f$ its homogenization $F\in \R[X_0,X_1,X_2]$, and the real algebraic surface defined by $S:=\{Z^2+F(X_0,X_1,X_2)=0\}$. He then interpretes the polynomials $f$ that are sums of three squares in $\R(x_1,x_2)$ as those for which the complex surface $S_{\C}$ carries an extra line bundle of a particular kind, and concludes by applying the Noether-Lefschetz theorem.

  As a consequence, $Q_d$ may be viewed as a union of Noether-Lefschetz loci in $P_d$. Over $\C$, density results for Noether-Lefschetz loci have been first obtained by Ciliberto-Harris-Miranda and Green \cite{CHM}, and we adapt these arguments over $\R$.

  We rely on a real analogue of Green's infinitesimal criterion (\cite[\S 5]{CHM}, \cite[\S 17.3.4]{voisinbook}), which was developed for other purposes in a joint work with Olivier Wittenberg \cite{BW}. Section \ref{sec1} is devoted to establishing this criterion in a form suitable for our needs: Proposition \ref{GreenRgeom}.
One way to verify the hypothesis of Green's criterion is to construct Noether-Lefschetz loci of the expected dimension. Following Ciliberto and Lopez \cite{CL}, we do so in Section \ref{sec2} by considering Noether-Lefschetz loci associated to determinantal curves, a strategy independently adopted by Bruzzo, Grassi and Lopez in \cite{BGL}. Finally, Section \ref{sec3} contains the proof of Theorem~\ref{main}.

\subsection{Level of function fields} 
\label{introlevel} The argument described above may be adapted to other families of real surfaces: here is another application of it.
  Recall that if $K$ is a field, Pfister \cite[Satz 4]{Pfisterlevel} has shown that the smallest integer $s$ such that $-1$ is a sum of $s$ squares in $K$ is a power of $2$ (or $+\infty$): it is the level $s(K)$ of $K$. Moreover, if $X$ is an integral variety over $\R$ of dimension $n$ without real points, $s(\R(X))\leq 2^n$ \cite[Theorem 2]{Pfisterreal}.

  Let us restrict to varieties that are smooth degree $d$ surfaces $S\subset\bP^3_{\R}$, defined by a degree $d$ homogeneous equation $F\in \R[X_0,X_1,X_2,X_3]_d$. Let $\Theta_d\subset\bP(\R[X_0,X_1,X_2,X_3]_d)$ be the set of those surfaces that have no real points. As before, we assume that $d$ is even, since otherwise $\Theta_d=\varnothing$.
 If $d=2$, any such surface $S$ is isomorphic to the anisotropic quadric $\{X_0^2+X_1^2+X_2^2+X_3^2=0\}$, and $s(\R(S))=2$. On the other hand, it follows from the Noether-Lefschetz theorem applied as in \cite{CTNL} that if $d\geq 4$, a very general $S\in \Theta_d$ satisfies $s(\R(S))=4$. In section \ref{sec4}, we will show:
\begin{thm}\label{main2}
The set of surfaces $S\in \Theta_d$ such that $s(\R(S))=2$ is dense in $\Theta_d$.
\end{thm}

\subsection{Conventions about algebraic varieties over $\R$}
\label{conv}
An algebraic variety $X$ over $\R$ is a separated scheme of finite type over $\R$. We denote its complexification by $X_{\C}$. Its set of complex points $X(\C)$ is endowed with an action of $G:=\Gal(\C/\R)\simeq\Z/2\Z$ such that the complex conjugation $\sigma \in G$ acts antiholomorphically.  The real points of $X$ are then the fixed points $X(\R)=X(\C)^G$.

Conversely, suppose that the set of complex points of a reduced quasi-projective algebraic variety $X_{\C}$ over $\C$  is endowed with an action of $G$ given by an antiholomorphic involution. If the isomorphism $X(\C)\simeq\overline{X(\C)}$ that the involution induces between $X(\C)$ and its conjugate variety is algebraic, Galois descent shows that $X_{\C}$  is naturally the complexification of an algebraic variety $X$ over $\R$.
We refer to \cite[I \S 1]{Silhol} or \cite[\S 2.2]{Mangolte} for more details. 

\section{Green's infinitesimal criterion over $\R$}\label{sec1}

This section is devoted to an adaptation over $\R$ of \cite[\S 5]{CHM}, where Green studies the relation between infinitesimal variations of Hodge structure and density of Noether-Lefschetz loci. We follow the exposition of \cite[\S 17.3.4]{voisinbook}. 

\subsection{Variations of Hodge structure over real varieties}
\label{VHSreel}

Let $B_{\C}$ be a smooth algebraic variety over $\C$, and
let $\mathbb{H}^2_{\Q}$ be a $\Q$-local system on $B(\C)$ carrying a weight $2$ variation of Hodge structure: in particular, the holomorphic vector bundle $\mathcal{H}^2:=\mathbb{H}^2_{\Q}\otimes_{\Q}\sO_{B(\C)}$ is endowed with a Hodge filtration by holomorphic subbundles $F^2\mathcal{H}^2\subset F^1\mathcal{H}^2\subset\mathcal{H}^2$, whose graded pieces $\mathcal{H}^{2,0}$, $\mathcal{H}^{1,1}$ and $\mathcal{H}^{0,2}$ may be viewed as $C^{\infty}$ complex subbundles of $\mathcal{H}^2$. Let  $\mathcal{H}^2_{\R}\subset\mathcal{H}^2$ and $\mathcal{H}^{1,1}_{\R}\subset\mathcal{H}^{1,1}$ be the $C^{\infty}$ real subbundles consisting of sections with values in the real local system $\mathbb{H}^2_{\R}:=\mathbb{H}^2_{\Q}\otimes_{\Q}\R$, and $\nabla:\mathcal{H}^2\to\mathcal{H}^2\otimes\Omega^1_{B(\C)}$ be the connection induced by $\mathbb{H}^2_{\Q}$. Finally, we still denote by $\mathcal{H}^2$ the total space of the geometric vector bundle associated to $\mathcal{H}^2$, whose fiber over $b\in B(\C)$ is $\mathcal{H}^2_b=\mathbb{H}^2_{\Q,b}\otimes_{\Q}\C$, and we use the same convention for its subbundles.

Assume now that $B_{\C}$ is the complexification of an algebraic variety $B$ over $\R$, and that we are given an action of $G$ on the $\Q$-local system $\mathbb{H}^2_{\Q}$ that is compatible with that on $B(\C)$. We consider the induced action of $G$ on (the total space of) $\mathcal{H}^2$ induced by the maps:
\begin{equation}
\label{actionantilin}
\mathcal{H}^2_b=\mathbb{H}^2_{\Q,b}\otimes_{\Q}\C\xrightarrow{\sigma\otimes\sigma} \mathbb{H}^2_{\Q,\sigma(b)}\otimes_{\Q}\C=\mathcal{H}^2_{\sigma(b)},
\end{equation}
where $\sigma$ acts naturally on the first factor and via complex conjugation on the second: in particular, $\sigma$ acts $\C$-antilinearly in the fibers of $\mathcal{H}^2$. We make the assumption that this action of $\sigma$ preserves the factors $\mathcal{H}^{p,q}$ 
of the Hodge decomposition. Consequently, there are induced actions of $G$ on $\mathcal{H}^{1,1}$ and $\mathcal{H}^{1,1}_{\R}$.

We will consider the $G$-module $\Z(1):=2\pi\sqrt{-1}\cdot\Z\subset\C$. It is
isomorphic to $\Z$ with an action of $G$ by multiplication by $-1$. We also denote by $\Z(1)$ the $G$-equivariant constant local system on $B(\C)$ with fibers $\Z(1)$, and if $M$ is any $G$-module or $G$-equivariant sheaf on $B(\C)$, we define $M(1)$ to be the tensor product $M\otimes_{\Z}\Z(1)$. If $b\in B(\R)$, we view $\mathcal{H}^{1,1}_{\R,b}(1)$ as a $G$-stable subspace of $\mathcal{H}^2_b$ via the embeddings $\mathcal{H}^{1,1}_{\R,b}\subset\mathcal{H}^2_b$ and $\Z(1)\subset \C$.

\subsection{The infinitesimal criterion}
\label{infcrit}

We fix $b\in B(\R)$, and we choose a $G$-stable connected analytic neighbourhood $\Delta$ of $b\in B(\C)$, on which $\mathbb{H}^2_{\Q}$ is trivialized, and such that $\Delta(\R):=\Delta\cap B(\R)$ is connected and contractible. 
Such a neighbourhood exists: choose any connected contractible neighbourhood of $b$ in $B(\C)$, intersect it with its image by $\sigma$, remove an appropriate closed subset of $\Delta(\R)$, and retain the connected component containing $b$. 

The trivialization of $\mathbb{H}^2_{\Q}$ over $\Delta$ gives rise to an isomorphism
$\left.\mathcal{H}^2\right|_{\Delta}\simeq \Delta\times \mathcal{H}^2_{b}$, that is $G$-equivariant by unicity of the trivialization.

\begin{prop}\label{GreenR}
Suppose that there exists $\lambda\in \mathcal{H}_{\R,b}^{1,1}(1)^G$ such that the map:
\begin{equation}
\label{infvar}
\overline{\nabla}(\lambda):T_{B(\C),b}
\longrightarrow\mathcal{H}_b^{0,2}
\end{equation}
 induced by evaluating the connection $\nabla$ on $\lambda$ is surjective.

 Then there exists an open cone $\Omega\subset \mathcal{H}^2_{\R,b}(1)^G$ such that for every $\omega\in\Omega$, there exists $b'\in\Delta(\R)$ and $\omega'\in \mathcal{H}^{1,1}_{\R,b'}(1)^G$ such that $\omega=\omega'$ under the
identification $\mathcal{H}^2_b\simeq \mathcal{H}^2_{b'}$ given by the trivialization.
\end{prop}

\begin{proof}
The trivialization of $\mathbb{H}^2_{\Q}$ over $\Delta$ yields an isomorphism
$\left.\mathcal{H}^2_{\R}\right|_{\Delta}\simeq \Delta\times \mathcal{H}^2_{\R,b}$.
Under our surjectivity hypothesis, the composition of the inclusion $\bigl.\mathcal{H}^{1,1}_{\R}\bigr|_{\Delta}\subset \bigl.\mathcal{H}^2_{\R}\bigr|_{\Delta}$ and of the projection to $\mathcal{H}^2_{\R,b}$ gives a map:
$$\phi:\left.\mathcal{H}^{1,1}_{\R}\right|_{\Delta}\to \mathcal{H}^2_{\R,b}$$
that is submersive at $(b,\lambda)$ by \cite[Lemme 17.21]{voisinbook}.
The map $\phi$ is equivariant with respect to the natural action of $G$ on both sides. Consequently, tensoring by $\Z(1)$ and taking $G$-invariants gives rise to a map:
$$\phi': \left.\mathcal{H}^{1,1}_{\R}(1)^G\right|_{\Delta(\R)}\to \mathcal{H}^2_{\R,b}(1)^G,$$
where $\left.\mathcal{H}^{1,1}_{\R}(1)^G\right|_{\Delta(\R)}$ is the $C^{\infty}$ real vector bundle on $\Delta(\R)$ with fiber $\mathcal{H}^{1,1}_{\R,b'}(1)^G$ at $b'\in\Delta(\R)$. 

It is well-known that the fixed locus $M^G$ of a $G$-action on a $C^{\infty}$ manifold $M$ is again a manifold, whose tangent space $T_x(M^G)$ at a fixed point $x\in M^G$ is $(T_xM)^G$ (see the much more general \cite[I \S 2.1]{Audin}). This implies that $\phi'$ is still submersive at $(b,\lambda)$. Consequently, the image of $\phi'$ contains an open set of $\mathcal{H}^2_{\R,b}(1)^G$. Since this image is obviously a cone, it contains an open cone $\Omega\subset \mathcal{H}^2_{\R,b}(1)^G$. This cone has the required property.
\end{proof}

\subsection{Density} 

In the classical complex case, the set of points in the base where Green's criterion may be verified is the complement of a complex-analytic subset. Consequently, assuming the base connected, it is dense if it is non-empty. The same holds in our setting if one restricts to a connected component of $B(\R)$, as one sees by adapting the argument to the real-analytic category.

\begin{prop}
\label{denseGreen}
Let $K\subset B(\R)$ be a connected component. The set of $b\in K$ for which there exists $\lambda\in  \mathcal{H}_{\R,b}^{1,1}(1)^G$ such that the map (\ref{infvar}) is surjective is either empty, or dense in $K$.
\end{prop}

\begin{proof}
We prove that the complement $W\subset K$ of this set is a real-analytic subset of $K$. This concludes because the set of $x\in K$ such that a neighbourhood of $x$ is included in $W$ is easily seen to be open and closed, hence equal to $K$ or empty, and it follows that the complement of $W$ is empty or dense in $K$, as wanted.

Since $F^1\mathcal{H}^2\subset \mathcal{H}^2$ is holomorphic and $\mathcal{H}^2_{\R}\subset\mathcal{H}^2$ is a real-analytic subbundle, we deduce that $\mathcal{H}^{1,1}_{\R}=F^1\mathcal{H}^2\cap\mathcal{H}^2_{\R}$ is a real-analytic subbundle of $\mathcal{H}^2$.
Since the action of $G$ on $B(\C)$ is real-analytic, and since the compatible action on $\mathcal{H}^2$ is induced by an action on $\mathbb{H}^2_{\Q}$, the $G$-action on $\mathcal{H}^2$ is real-analytic, hence so is the subbundle $\mathcal{H}^{1,1}_{\R}(1)^G\subset\left.\mathcal{H}^2\right|_{B(\R)}$. Since the connection $\nabla$ induces a morphism $\overline{\nabla}:\mathcal{H}^{1,1}\otimes T_{B(\C)}\to\mathcal{H}^{0,2}$ of holomorphic vector bundles \cite[(10.2.1)]{voisinbook},
the set $Z\subset\mathcal{H}^{1,1}_{\R}(1)^G$ consisting of the $(b,\lambda)$  for which 
the map (\ref{infvar})
is not surjective is a real-analytic subset of $\mathcal{H}^{1,1}_{\R}(1)^G$.
 
We may deduce that $W:=\{b\in B(\R)\mid \mathcal{H}^{1,1}_{\R,b}(1)^G\subset Z\}$ 
 is a real-analytic subset of $B(\R)$. To check it locally at $b\in B(\R)$, choose a trivialization of $\mathcal{H}_{\R}^{1,1}(1)^G$ in a neighbourhood $U\subset B(\R)$ of $b$ with fiber $\R^k$ and notice that $W\cap U$ is real-analytic as the intersection of the real-analytic subsets $(Z\cap (U\times\{v\})_{v\in\R^k}$ of $U$.
\end{proof}

\subsection{Families of real varieties} 
\label{famreel}

Let us specialize Propositions \ref{GreenR} and \ref{denseGreen} to variations of Hodge structure of geometric origin.

Let $\pi:\mathcal{S}\to B$ be a smooth projective morphism of smooth algebraic varieties over $\R$: both $\mathcal{S}(\C)$ and $B(\C)$ are endowed with an action of $G$ such that $\sigma $ acts antiholomorphically, and the map $\pi:\mathcal{S}(\C)\to B(\C)$ is $G$-equivariant.

The local system $\mathbb{H}^2_{\Q}:=R^2\pi_{*}\Q$ on $B(\C)$ underlies a weight $2$ variation of Hodge structure, whose connection is the Gauss-Manin connection. The action of $G$ on $\mathcal{S}(\C)$ induces an action of $G$ on $\mathbb{H}^2_{\Q}$ that is compatible with the $G$-action on $B(\C)$.
Let us verify the assumption made in \S\ref{VHSreel} that the induced $\C$-antilinear action on $\mathcal{H}^2$ preserves the Hodge decomposition. 
By (\ref{actionantilin}), this action may be written as a composition:
\begin{equation}
\mathcal{H}^2_b=\mathbb{H}^2_{\Q,b}\otimes_{\Q}\C\xrightarrow{\Id\otimes\sigma}
\mathbb{H}^2_{\Q,b}\otimes_{\Q}\C\xrightarrow{\sigma\otimes\Id}\mathbb{H}^2_{\Q,\sigma(b)}\otimes_{\Q}\C=\mathcal{H}^2_{\sigma(b)}.
\end{equation}
The first arrow is the conjugation with respect to the real structure $\mathcal{H}^2_{\R,b}\subset\mathcal{H}^2_b$, hence exchanges the factors $\mathcal{H}_b^{p,q}$ and $\mathcal{H}_b^{q,p}$ of the Hodge decomposition. The second arrow is obtained by functoriality from the antiholomorphic map $\sigma: \mathcal{S}_{\sigma(b)}(\C)\to\mathcal{S}_b(\C)$, hence also exchanges the factors of the Hodge decomposition by the argument of \cite[I Lemma 2.4]{Silhol}.
Consequently, we are indeed in the setting of \S\ref{VHSreel}. 

Let $b\in B(\R)$ and $b\in\Delta\subset B(\C)$ be as in \S\ref{infcrit}. Griffiths \cite[Th\'eor\`eme 10.21]{voisinbook} has computed the map $\overline{\nabla}(\lambda)$ of (\ref{infvar}) as the composition of the Kodaira-Spencer map and of the contracted cup-product with $\lambda$ induced by the pairing $\Omega^1_{\mathcal{S}_{b,\C}}\otimes T_{\mathcal{S}_{b,\C}}\to\sO_{\mathcal{S}_{b,\C}}$:
\begin{equation}
\label{ksGriff}
T_{B_{\C},b}\longrightarrow H^1(\mathcal{S}_{b,\C},T_{\mathcal{S}_{b,\C}})\stackrel{\lambda}\longrightarrow H^2(\mathcal{S}_{b,\C}, \sO_{\mathcal{S}_{b,\C}}).
\end{equation}
Propositions \ref{GreenR} and \ref{denseGreen} then become: 

\begin{prop}
\label{GreenRgeom}
Suppose that there exists $\lambda\in H^{1,1}_{\R}(\mathcal{S}_b(\C))(1)^G$ such that the composition (\ref{ksGriff})
is surjective.
 Then there exists an open cone $\Omega\subset H^2(\mathcal{S}_b(\C),\R(1))^G$ such that for every $\omega\in\Omega$, there exists $b'\in\Delta(\R)$ and $\omega'\in H^{1,1}_{\R}(\mathcal{S}_{b'}(\C))(1)^G$ such that $\omega=\omega'$ under the
identification $H^2(\mathcal{S}_b(\C),\C)\simeq H^2(\mathcal{S}_{b'}(\C),\C)$ given by the trivialization.

Moreover, the set of $b\in B(\R)$ for which there exists such a $\lambda$ is dense in every connected component $K$ of $B(\R)$ that it meets.
\end{prop}

\begin{rem}
\label{algcohoclass}
In order to verify the hypothesis of Proposition \ref{GreenRgeom}, one has to construct a class $\lambda\in H^{1,1}_{\R}(\mathcal{S}_b(\C))(1)^G$. A natural source of such cohomology classes are cycle classes $[\mathcal{L}]\in H^2(\mathcal{S}_b(\C),\R(1))$ of line bundles $\mathcal{L}$ on $\mathcal{S}_b$ that are defined over $\R$. Indeed, $[\mathcal{L}]$ is of type $(1,1)$ by Hodge theory, and analyzing the $G$-action on the exponential exact sequence  \cite[I (4.11), Lemma 4.12]{Silhol} shows that it belongs to $H^2(\mathcal{S}_b(\C),\R(1))^G$.
\end{rem}

\begin{rem}
It is possible, in the setting of Proposition \ref{GreenRgeom}, that the set $\Sigma$ of $b\in B(\R)$ such that there exists a $\lambda\in H^{1,1}_{\R}(\mathcal{S}_b(\C))(1)^G$ for which (\ref{ksGriff}) is surjective is dense in some connected component of $B(\R)$, but does not meet another one.

This happens when $\mathcal{S}\to B$ is the universal family of smooth quartic surfaces in $\bP^3_{\R}$. In this case, we will show in Lemma \ref{dense2} that $\Sigma$ is dense in the connected components of $B(\R)$ parameterizing surfaces without real points. However, $\Sigma$ cannot intersect the connected components corresponding to surfaces $S$ whose real locus is a union of $10$ spheres (that exist by \cite[(3.3) p.189]{Silhol}). Indeed, one computes using \cite[VIII \S 3]{Silhol} that for such surfaces, $H^{1,1}_{\R}(S(\C))(1)^G$ has rank $1$, hence is generated by $\lambda:=[\sO_{\bP^3_{\R}}(1)]$. Since the line bundle $\sO_{\bP^3_{\R}}(1)$ is defined on the whole family, its cohomology class $\lambda$ remains Hodge under small deformations showing that (\ref{ksGriff}) vanishes.
\end{rem}

\begin{rem}
If $b\in B(\R)$, let $F_{\infty}:=\sigma\otimes\Id$ be the $\C$-linear involution of $\mathcal{H}^2_b=\mathbb{H}^2_{\Q,b}\otimes_{\Q}\C$ induced by $\sigma$. It exchanges the factors of the Hodge decomposition, hence preserves $\mathcal{H}^{1,1}_b$. Since $\mathcal{H}^{1,1}_{\R,b}(1)^G$ is Zariski-dense in its complexification $(\mathcal{H}^{1,1}_{b})^{F_{\infty}=-1}$, and since the surjectivity of (\ref{ksGriff}) depends algebraically on $\lambda$, it would be sufficient, in the hypotheses of Proposition \ref{GreenRgeom}, to require that $\lambda\in (\mathcal{H}^{1,1}_b)^{F_{\infty}=-1}$. 
\end{rem}

\begin{rem}
If the family $\mathcal{S}\to B$ is induced by a linear system in a fixed variety $X$ over $\R$, it may be better to apply Proposition \ref{GreenR} to the variation of Hodge structure given by the vanishing cohomology. 
\end{rem}

\section{An explicit Noether-Lefschetz locus}
\label{sec2}

To verify the hypothesis of Proposition \ref{GreenRgeom}, we need to construct an appropriate cohomology class $\lambda$. In the complex setting, several strategies are available to do so: the original degeneration method of Ciliberto-Harris-Miranda \cite{CHM}, computations with jacobian rings \cite[Theorem 2]{Kim}, use of explicit Noether-Lefschetz loci \cite{CL}, or the much more general arguments of Voisin \cite{voisin}.

Here, we adapt the strategy of Ciliberto and Lopez \cite[Lemma 1.2, Theorem 1.3 and their proofs]{CL}: we take for $\lambda$ the class of a determinantal curve. This section is devoted to working out this idea in the generality we need. We first analyze Green's criterion when $\lambda$ is the class of a curve in \S\ref{Greencurve}, and specialize to the case of a determinantal curve in \S\ref{determinantal}. The main difference with \cite{CL} is that we argue purely cohomologically rather than geometrically on the Noether-Lefschetz loci.

\subsection{Applying Green's criterion to the class of a curve} 
\label{Greencurve}In this paragraph, we fix smooth projective complex varieties $C\subset S\subset X$, where $C$ is a curve, $S$ a surface and $X$ a threefold. The image in $H^2(S(\C),\C)$ of the Betti cohomology class $\lambda\in H^2(S(\C),\Z(1))$ of $C$ in $S$ is of type $(1,1)$ by Hodge theory. We may thus view it as an element $\lambda\in H^1(S,\Omega^1_S)=H^{1,1}(S)$.
We study the composition:
\begin{equation}
\label{ksmap}
\psi_{\lambda}:H^0(S,N_{S/X})\to H^1(S, T_{S})\stackrel{\lambda}\longrightarrow H^2(S, \sO_{S})
\end{equation}
of the boundary map of the normal exact sequence $0\to T_S\to \left.T_X\right|_S\to N_{S/X}\to 0$
and of the contracted cup-product with $\lambda$.
To do so, we consider the two exact sequences:
\begin{equation}
\label{ses1}
0\to N_{C/S}\to N_{C/X}\to \left.N_{S/X}\right|_C\to 0,
\end{equation}
\begin{equation}
\label{ses2}
0\to \sO_S \to \sO_S(C)\to \left.\sO_S(C)\right|_C\to 0,
\end{equation}
and we recall that $N_{C/S}\simeq\left.\sO_S(C)\right|_C$.

\begin{prop}
\label{composition}
The map $\psi_{\lambda}$ of (\ref{ksmap}) coincides with the composition:
$$H^0(S,N_{S/X})\to H^0(C,\left.N_{S/X}\right|_C)\to H^1(C, N_{C/S})\to H^2(S,\sO_S)$$
 of the restriction map and of the boundary maps of (\ref{ses1}) and (\ref{ses2}).
\end{prop}
We first recall some properties of extension classes used in the proof. Let 
\begin{equation}
\label{extension}
0\to\mathcal{A}\xrightarrow{f}\mathcal{B}\xrightarrow{g}\mathcal{C}\to 0
\end{equation}
 be an exact sequence of locally free sheaves on $S$. Choose an open cover $(U_i)$ of $S$  such that $g|_{U_i}$ admits a section $\sigma_i:\mathcal{C}|_{U_i}\to \mathcal{B}|_{U_i}$. Setting $\tau_{i,j}:=\sigma_j-\sigma_i\in H^0(U_i\cap U_j,\mathcal{A}\otimes\mathcal{C}^\vee)$ gives rise to a cocycle $(\tau_{i,j})$ whose  cohomology class $\xi\in H^1(S,\mathcal{A}\otimes\mathcal{C}^\vee)$ is independent of the choices. It is the extension class of (\ref{extension}).
Direct computations with cocycles show that the extension class of the tensor product $0\to\mathcal{A}\otimes\mathcal{L}\xrightarrow{f}\mathcal{B}\otimes\mathcal{L}\xrightarrow{g}\mathcal{C}\otimes\mathcal{L}\to 0$  by a line bundle $\mathcal{L}$ on $S$ is equal to $\xi$, that the extension class of the dual $0\to\mathcal{C}^\vee\xrightarrow{g^\vee}\mathcal{B}^\vee\xrightarrow{f^\vee}\mathcal{A}^\vee\to 0$ is equal to $-\xi$, and that the boundary maps $H^q(S,\mathcal{C})\to H^{q+1}(S,\mathcal{A})$ in the long exact sequence of cohomology associated to (\ref{extension}) are induced by the cup-product by $\xi$.
 
\begin{proof}[Proof of Proposition \ref{composition}]
The proposition follows from the compatibility of the two boundary maps appearing in the commutative diagram of coherent sheaves on $S$:
\begin{align}
\label{diag1}
\begin{aligned}
\xymatrix@R=3ex{
0\ar[r] & T_S \ar[r] \ar[d] & \left.T_X\right|_S \ar[d] \ar[r] & N_{S/X}\ar[d]\ar[r] & 0 \\
0\ar[r] & N_{C/S} \ar[r] & N_{C/X} \ar[r] &\left.N_{S/X}\right|_C\ar[r] & 0,
}
\end{aligned}
\end{align}
and in the pull-back diagram:
\begin{align}
\label{diag2}
\begin{aligned}
\xymatrix@R=3ex{
0\ar[r] & \sO_S\ar[r] \ar@{=}[d] & \mathcal{E}\ar[r]\ar[d] & T_S\ar[d]\ar[r] & 0 \\
0\ar[r] & \sO_S \ar[r] & \sO_S(C) \ar[r] &N_{C/S}\ar[r] & 0,
}
\end{aligned}
\end{align}
in addition to the fact that the boundary map $H^1(S,T_S)\to H^2(S,\sO_S)$ associated to the first line of (\ref{diag2}) is induced by the cup-product with its extension class, that turns out to be equal to $\lambda$. 
To verify this fact, we rather consider the twist of (\ref{diag2}) by $\sO_S(-C)$:
\begin{align*}
\label{diag2bis}
\begin{aligned}
\xymatrix@R=3ex{
0\ar[r] & \sO_S(-C)\ar[r] \ar@{=}[d] & \mathcal{E}(-C)\ar[r]\ar[d] & T_S(-C)\ar[d]\ar[r] & 0 \\
0\ar[r] & \sO_S(-C) \ar[r] & \sO_S \ar[r] &\sO_C\ar[r] & 0,
}
\end{aligned}
\end{align*}
and we prove that the extension on the first line is canonically dual to the extension: $$0\to\Omega^1_S(C)\to D(\sO_S(C))\to \sO_S(C)\to 0$$ defined by Atiyah in \cite[\S 4]{Connections}, and whose extension class (the Atiyah class) is equal to $-\lambda$ by \cite[Proposition 12]{Connections} (the factor $2\pi i$ in loc. cit. corresponds to the comparison between Chern classes in Betti and de Rham cohomology, and is accounted for here by our definition of $\Z(1)$). 

To check this duality statement, choose $x\in S$ and local sections $f\in \sO_{S,x}$ and $v\in T_S(-C)_x$ that coincide in $\sO_{C,x}$
hence induce $(f,v)\in\mathcal{E}(-C)_x$. Let $s\in\sO_S(C)_x$ and $\beta\in \Omega^1_S(C)_x$ giving rise to a local section $(s,\beta)\in D(\sO_S(C))_x$ in the notations of  \cite[\S 4]{Connections}. A direct computation shows that $(f,v)\cdot(s,\beta)=\beta(v)+sf-ds(v)$ is well-defined, $\sO_{S,x}$-linear, and induces the required duality.
\end{proof}

\begin{cor}
\label{coroGreenapplies}
If the groups $H^1(S,N_{S/X}(-C))$, $H^1(C,N_{C/X})$ and $H^2(S,\sO_S(C))$ vanish, the map $\psi_{\lambda}$ of (\ref{ksmap}) is surjective.
\end{cor}

\begin{proof}
Using (\ref{ses1}) and (\ref{ses2}), this follows from Proposition \ref{composition}.
\end{proof}

\subsection{The case of a determinantal curve}
\label{determinantal}
We now restrict the situation to the case where $C$ is a determinantal curve. Let $X$ be a smooth projective connected complex threefold, and $n\geq 1$ be an integer.
Let $L, H_1,\dots, H_n$ be base-point free line bundles on $X$, and define $H:=H_1+\dots+H_n$. 
In this paragraph, we make the following assumptions:
\begin{hyp}
\label{hypo}
For every $1\leq j,k\leq n$:
\begin{enumerate}[(i)]
\item $H^2(X,\sO_X)=0$,
\item $H^1(X,H_j)=H^2(X,H_j-H_k)=H^3(X,-H_j)=0$,
\item $H^0(X,K_X+L)=H^1(X,L)=H^1(X,K_X+L-H_j)=H^2(X,L-H_j)=0$.
\end{enumerate}
\end{hyp}
The particular case considered in \cite{CL} is $X=\bP^3_{\C}$, $H_j=\sO_{\bP^3_{\C}}(1)$ and $L=\sO_{\bP^3_{\C}}(2)$.

  Choose $M=(M_{i,j})_{1\leq i\leq n+1, 1\leq j\leq n}$ a $(n+1)\times n$ matrix with $M_{i,j}\in H^0(X,H_j)$.
Let $C\subset X$ be defined by the vanishing of the maximal minors of $M$, and $S\subset X$ be the zero-locus of a section $\tau\in H^0(X,H+L)$ vanishing on $C$.

\begin{lem}\label{smooth}
If $M$ is general, $C$ is a smooth curve, possibly empty. Once such a matrix $M$ is fixed, if $\tau$ is general, $S$ is a smooth surface, possibly empty. 
\end{lem}

\begin{proof}
Let $\mathcal{M}:=\bigoplus_{1\leq i\leq n+1, 1\leq j\leq n}H^0(X,H_j)$ be the parameter space for such matrices, and let $\mathcal{C}\subset\mathcal{M}\times X$ be the universal variety defined by the vanishing of maximal minors. We consider the fiber $\mathcal{C}_x$ of the second projection $\mathcal{C}\to X$ at $x\in X(\C)$. Since the $H_j$ are base-point free, evaluation at $x$ with respect to local trivializations yields a linear surjection $\ev_x:\mathcal{M}\to M_{n+1,n}(\C)$, and $\mathcal{C}_x=\{M\in \mathcal{M}\mid \rank(\ev_x(M))<n\}$.
It follows that $\mathcal{C}_x\subset\mathcal{M}$ is irreducible of codimension $2$ and that its singular locus $\Sing(\mathcal{C}_x)\subset\mathcal{M}$, being the inverse image by $\ev_x$ of the the set of matrices of rank $<n-1$, has codimension $6$ \cite[Proposition 1.1]{Determinantal}. We deduce that $\mathcal{C}$ is irreductible of dimension $\dim(\mathcal{M})+1$ and that $\dim(\Sing(\mathcal{C}))<\dim(\mathcal{M})$. It follows that the generic fiber of the first projection $\mathcal{C}\to \mathcal{M}$ has dimension $1$ (but may be empty if this projection is not dominant) and does not meet $\Sing(\mathcal{C})$, hence is smooth by generic smoothness as we are in characteristic $0$. Consequently, we may choose $M\in\mathcal{M}$ so that $C$ is a smooth curve.

That $S$ may be chosen smooth follows from an easy variant of the results of \cite{AK}.
Since $C$ is defined by the vanishing of sections of $H$, $\mathcal{I}_C(H)$ is generated by its global sections, and since $L$ is base-point free, so is $\mathcal{I}_C(H+L)$. Fix $x\in C$, and let $\mathfrak{m}_x\subset\sO_{X,x}$ be the maximal ideal.  As $C$ is a smooth curve, the evaluation map $H^0(X,\mathcal{I}_C(H+L))\to \sO_X(H+L)\otimes(\mathfrak{m}_x/\mathfrak{m}_x^2)$ has rank $2$, so that the set of $\tau\in H^0(X,\mathcal{I}_C(H+L))$ whose zero-locus is singular at $x$ has codimension $2$. A dimension count shows that the zero-locus of a general $\tau$ is smooth along $C$. The base locus of the linear system $H^0(X,\mathcal{I}_C(H+L))$ being $C$, it follows from Bertini's theorem \cite[I Th\'eor\`eme 6.10 2)]{Jouanolou} that the zero-locus of a general $\tau$ is smooth off $C$. We deduce that if $\tau$ is general, $S$ is smooth.
\end{proof}

From now on, we suppose that $M$ and $\tau$ have been chosen general in the sense of Lemma \ref{smooth}. Our goal is to show in Proposition \ref{surj} below that, under Hypotheses \ref{hypo}, the morphism $\psi_{\lambda}$ defined in (\ref{ksmap}) is surjective. We first explain the tools that will allow to carry out the relevant coherent cohomology computations.

 By Lemma \ref{smooth}, the curve $C$ is a determinantal curve of the expected codimension in $X$. It follows that its ideal sheaf $\mathcal{I}_C\subset\sO_X$ is resolved by the Eagon-Northcott complex (also called the Hilbert-Burch complex in this particular case, see \cite[Theorem A2.60, Example A2.67]{syzygies}):
\begin{equation}
\label{EN}
0\to \bigoplus_{j=1}^n \sO_X(-H-H_j)\stackrel{M}\longrightarrow \sO_X(-H)^{\oplus n+1}\to\mathcal{I}_C\to 0,
\end{equation}
in which the first map is given by the matrix $M$ and the second one by the $n+1$ maximal minors of $M$.
Restricting (\ref{EN}) to $C$ using right exactness of the tensor product, and noticing that the kernel of $\bigoplus_{j=1}^n \sO_C(-H-H_j)\stackrel{M}\longrightarrow \sO_C(-H)^{\oplus n+1}$ is a line bundle on $C$ that may be computed by calculating its determinant, one gets:
\begin{equation}
\label{ENonC}
0\to  K_C^{-1}\otimes\left.K_X\right|_C\to \bigoplus_{j=1}^n \sO_C(-H-H_j)\stackrel{M}\longrightarrow \sO_C(-H)^{\oplus n+1}\to N_{C/X}^{\vee}\to 0.
\end{equation}
The dual of the Eagon-Northcott complex is still a resolution by \cite[Theorem A2.60]{syzygies}, of a sheaf on $X$ that we denote by $\mathcal{Q}$:
\begin{equation}
\label{ENdual}
0\to\sO_X\to \sO_X(H)^{\oplus n+1}\stackrel{M^\mathsf{T}}\longrightarrow  \bigoplus_{j=1}^n \sO_X(H+H_j)\to \mathcal{Q}\to 0.
\end{equation}
 It follows from Cramer's rule that the maximal minors of $M$ vanish on the support of $\mathcal{Q}$. Consequently, $\mathcal{Q}$ may be computed after restriction to $C$, and the dual of (\ref{ENonC}) shows that $\mathcal{Q}=K_C\otimes \left.K_X^{-1}\right|_C$.  Finally, there is an obvious morphism between (\ref{ENdual}) and the dual of (\ref{ENonC}), where the left vertical arrow is the zero map:

\begin{align}
\label{ENdiagram}
\begin{aligned}
\xymatrix@R=3ex@C=1em{
0\ar[r] & \sO_X \ar[r] \ar[d]^0 &\sO_X(H)^{\oplus n+1} \ar[d] \ar[r] &  \bigoplus_{j=1}^n \sO_X(H+H_j)\ar[d]\ar[r] &K_C\otimes \left.K_X^{-1}\right|_C\ar[r] \ar@{=}[d]& 0 \\
0\ar[r] & N_{C/X} \ar[r] & \sO_C(H)^{\oplus n+1} \ar[r] & \bigoplus_{j=1}^n \sO_C(H+H_j)\ar[r] &K_C\otimes \left.K_X^{-1}\right|_C\ar[r]& 0.
}
\end{aligned}
\end{align}

\begin{prop}
\label{surj}
The map $\psi_{\lambda}$ of (\ref{ksmap}) is surjective.
\end{prop}

\begin{proof}
By Corollary \ref{coroGreenapplies}, it suffices to show the vanishing of the three groups $H^1(S,N_{S/X}(-C))$, $H^1(C,N_{C/X})$ and $H^2(S,\sO_S(C))$.
Twisting the natural exact sequence:
\begin{equation}
\label{ideals}
0\to\sO_X(-S)\to\mathcal{I}_C\to\sO_S(-C)\to 0
\end{equation}
by $\sO_X(S)$ and taking cohomology, we see that the vanishing of $H^1(S,N_{S/X}(-C))$ follows from that of $H^2(X,\sO_X)$ and $H^1(X,\mathcal{I}_C(S))$ which are deduced from Hypotheses \ref{hypo} using (\ref{EN}).

  A diagram chase using (\ref{ENdiagram}) (or, more conceptually, an analysis of the second hypercohomology spectral sequence of this exact sequence of complexes) shows that in order to prove the vanishing of $H^1(C,N_{C/X})$, it suffices to check that $H^2(X,\mathcal{I}_C(H))=H^1(X,\mathcal{I}_C(H+H_j))=0$ for every $1\leq j\leq n$. In turn, these vanishings follow from the Hypotheses \ref{hypo} and from (\ref{EN}).

By Serre duality and adjunction, the vanishing of $H^2(S,\sO_S(C))$ is equi\-valent to that of $H^0(S,K_S(-C))=H^0(S,(K_X+S)|_S(-C))$. To prove it, twist (\ref{ideals}) by $K_X+S$, take cohomology, and notice that $H^1(X,K_X)=H^0(X,\mathcal{I}_C(K_X+S))=0$ by (\ref{EN}) and Hypotheses \ref{hypo}.
\end{proof}

\begin{rem}
In this situation, the Kodaira-Spencer map $H^0(S,N_{S/X})\to H^1(S, T_{S})$ appearing in (\ref{ksmap}) may itself not be surjective. This will be the case when we apply Proposition \ref{surj} in \S\ref{Greendouble}.
\end{rem}

\section{Sums of three squares}\label{sec3}

In this section, we prove Theorem \ref{main}.
We fix an even integer $d=2\delta$ with $\delta\geq 1$.

We explain in \S\ref{sosandNL} the connection between sums of squares in $\R(x_1,x_2)$ and line budles on double covers of $\bP^2_{\R}$, relating our problem to the study of Noether-Lefschetz loci. In \S \ref{setup}--\ref{Greendouble}, we apply the results of Section \ref{sec2} to verify Green's infinitesimal criterion for the family of real double covers of $\bP^2_{\R}$, and the proof of Theorem \ref{main} is completed in \S\ref{parproof}--\ref{proof}.

\subsection{Sums of squares and line bundles}
\label{sosandNL}
We first recall two lemmas already
used by Colliot-Th\'el\`ene \cite{CTNL}.

\begin{lem}
\label{23}
Let $K$ be a field of characteristic $\neq 2$, and $f\in K$, and $L:=K[\sqrt{-f}]$. Then the following are equivalent:\begin{enumerate}[(i)]
\item $f$ is a sum of three squares in $K$.
\item $-1$ is a sum of two squares in $L$.
\end{enumerate}
\end{lem}

\begin{proof}
This is \cite[Lemma 1.2]{CTNL} (see also \cite[Chap. 11 Theorem 2.7]{Lam}).
\end{proof}

\begin{lem}
\label{NLsquares}
Let $S$ be a smooth projective geometrically connected variety over $\R$.
Then the following are equivalent:
\begin{enumerate}[(i)]
\item $-1$ is a sum of two squares in $\R(S)$.
\item the pull-back map $\Pic(S)\to\Pic(S_{\C})^G$ is not surjective.
\end{enumerate}
\end{lem}

\begin{proof}
This is \cite[Chapter I, Corollary 2.5]{vanhamelthese}. For the convenience of the reader, we recall the argument. Condition (i) is equivalent to the nontrivial quaternion algebra over $\R$ splitting over $\R(S)$ \cite[Proposition 1.1.7]{GilleSzamuely}, hence to the pull-back map $\Z/2\Z\simeq\Br(\R)\to\Br(\R(S))$ being zero. The exact sequence \cite[Lemma 1.1]{CTNL}:
$$0\to \Pic(S)\to\Pic(S_{\C})^G\to \Br(\R)\to\Br(\R(S)),$$
shows that this is equivalent to (ii).
\end{proof}

In \S\ref{parproof}, Lemmas \ref{23} and \ref{NLsquares} will be applied to a positive semidefinite polynomial $f\in K=\R(x_1,x_2)$, and to the quadratic extension $L=K[\sqrt{-f}]=\R(S)$ associated to the double cover $S\to\bP^2_{\R}$ determined by $f$.

\subsection{A real double cover containing a determinantal curve} 
\label{setup}
In this paragraph, we construct varieties $C$, $S$, and $X$ to which the results of \S\ref{determinantal} apply.

Let $\Gamma:=\{X_0^2+X_1^2+X_2^2=0\}\subset\bP^2_{\R}$ be the anisotropic conic over $\R$. There is an isomorphism between its complexification $\Gamma_{\C}$ and the projective line $\bP^1_{\C}$.
The line bundle $\sL:=\sO_{\Gamma_{\C}}(1)$ is however not defined over $\R$ for the real structure we consider on $\Gamma_{\C}$ (as the zero locus of a real section of $\mathcal{L}$ would be a real point of $\Gamma$). 
Instead, it has a so-called quaternionic structure: it may be equipped with an isomorphism $\phi:\sL\stackrel{\sim}\longrightarrow\sigma^*\overline{\mathcal{L}}$ such that $\phi\circ(\sigma^*\overline{\phi})=-\Id$, where $\overline{\sL}$ denotes the conjugate line bundle on the conjugate variety $\overline{\Gamma_{\C}}$, and the real structure $\Gamma$ of $\Gamma_{\C}$ is viewed as an isomorphism $\sigma:\Gamma_{\C}\stackrel{\sim}\longrightarrow\overline{\Gamma_{\C}}$ (see for instance \cite[\S 3]{BHH}). The isomorphism $\phi$ induces a $\C$-antilinear automorphism $\sigma^*$ of $H^0(\Gamma_{\C},\sL)$ such that $\sigma^*\circ\sigma^*=-1$. One may then choose a basis $(A,B)$ of $H^0(\Gamma_{\C},\sL)$ 
on which the action of $\sigma^*$  is given by:
\begin{equation}
\label{quaternionicaction}
 \sigma^*A=B\textrm{ and }\sigma^*B=-A.
\end{equation}
In view of (\ref{quaternionicaction}), the isomorphism $\Gamma_{\C}\xrightarrow{\sim}\bP^1_{\C}$ defined by $x\mapsto[A(x):B(x)]$ induces an action of $G$ on $\bP^1(\C)$ given by 
 $\sigma([a:b])=[\overline{b}:-\overline{a}]$. 

We define $Y:=\Gamma\times\bP^2_{\R}$ and $X:=Y_{\C}\simeq\Gamma_{\C}\times\bP^2_{\C}$ to be the complexification of $Y$ with the induced real structure. 
We set $n:=2+\delta$, and introduce the following line bundles on $X$: $H_j:=p_1^*\sO_{\Gamma_{\C}}(1)$ if $j=1,2$, $H_j=p_2^*\sO_{\bP^2_{\C}}(1)$ if $j\in\{3,\dots,2+\delta\}$, $H=\sum_j H_j$ and $L=\sO_X$. Note that Hypotheses \ref{hypo} are satisfied. 

Let $\bP$ be the parameter space for $(n+1)\times n$ matrices $M=(M_{i,j})_{1\leq i\leq n+1, 1\leq j\leq n}$ with $M_{i,j}\in H^0(X,H_j)$ such that no column is identically zero, and whose columns are well-defined up to multiplication by a scalar: $\bP$ is a product of $n$ projective spaces.
To such a matrix $M$, we associate the variety $C\subset X$ defined by the vanishing of all the maximal minors of $M$.
Let $\tau\in H^0(X,H)$ be a section vanishing on $C$, and $S\subset X$ be the zero-locus of $\tau$.

\begin{lem}
\label{conditionssurface}
There exist $M$ and $\tau$ such that:
\begin{enumerate}[(i)]
\item $C$ is a smooth curve and $S$ is a smooth surface,
\item the projection
 $\left.p_2\right|_S:
S\to\bP^2_{\C}$ is a finite double cover ramified along a smooth degree $d$ curve  $D\subset\bP^2_{\C}$,
\item the subvarieties $C$ and $S$ of $X$ are defined over $\R$.
\end{enumerate}
\end{lem}

\begin{proof}
Lemma \ref{smooth} shows that there is a non-empty Zariski-open subset of $\bP$ over which $C$ is smooth, and that for a generic choice of $\tau$, $S$ is also smooth.

Let us show that if $M$ is general, the projection $S\to\bP^2_{\C}$ is finite when $\tau$ is the first maximal minor $\det(M_{i,j})_{1\leq i,j\leq n}$ of $M$ (and consequently when $\tau$ is general). It suffices to exhibit one such $M$, for which we use homogeneous coordinates $A,B$ on $\Gamma_{\C}$ as above and $X_0,X_1,X_2$ on $\bP^2_{\C}$.
 One can take
$$M=\begin{pmatrix}
A & 0 &X_0 \\
B & A &X_1 \\
0 & B & X_2
\end{pmatrix}$$
when $\delta=1$. One verifies that a general matrix whose first and second columns are $(A,B,0,\dots,0)$ and $(0,0,A,B,0,\dots,0)$ works when $\delta\geq 2$.

We have shown the existence of a non-empty Zariski-open subset $U\subset\bP$ for which $C$ is a smooth curve, and such that $S$ is smooth with finite projection $S\to\bP^2_{\C}$ for a general choice of $\tau$.
If $P(X_0,X_1,X_2)A^2+Q(X_0,X_1,X_2)AB+R(X_0,X_1,X_2)B^2=0$ is the equation in $X$ of such a surface $S$, where $P,Q$ and $R$ have degree $\delta$, a direct computation shows that the projection $S\to\bP^2_{\C}$ is finite of degree $2$ with ramification locus $D\subset\bP^2_{\C}$ defined by the equation $Q^2=4PR$ of degree $2\delta=d$, and that the smoothness of $S$ implies that of $D$.

Let $U'\subset \bP$ be the open set consisting of matrices whose first two columns are not proportional, and notice that $U\subset U'$. There are fixed point free actions of $\Z/2\Z$ on $U'$ and $U$ obtained by exchanging the first two columns. The quotients are smooth complex algebraic varieties $U/(\Z/2\Z)\subset U'/(\Z/2\Z)$.

  Letting $\sigma$ act on $H^0(\bP^2_{\C},\sO_{\bP^2_{\C}}(1))$ using the natural real structure $\sO_{\bP^2_{\R}}(1)$ of  $\sO_{\bP^2_{\C}}(1)$ and on $H^0(\bP^1_{\C},\sO_{\bP^1_{\C}}(1))$ using (\ref{quaternionicaction}), we obtain a $G$-action on $\bP$, which descends to a $G$-action on $U/(\Z/2\Z)$ and $U'/(\Z/2\Z)$, endowing these algebraic varieties over $\C$ with a real structure by \S\ref{conv}. It is obvious that $U'/(\Z/2\Z)$ has a real point for this real structure (for instance, choose $M_{1,1}=A$, $M_{1,2}=B$, and $M_{i,j}=0$ otherwise). Since $U'/(\Z/2\Z)$ is smooth and irreducible, the implicit function theorem shows that its real points are Zariski-dense \cite[Proposition 1.1]{H17}, 
and we deduce that $U/(\Z/2\Z)$ also has a real point. Choose $M$ to be a matrix lifting this real point: then $C\subset X$ is defined over $\R$. Finally, choose $\tau$ general and defined over $\R$ to ensure that $S\subset X$ is also defined over $\R$.
\end{proof}

\begin{rem}
A variant of our strategy would have been to choose $n=1+\delta$, $H_1:=p_1^*\sO_{\bP^1_{\C}}(2)$ (which has the advantage of being defined over $\R$), $H_j=p_2^*\sO_{\bP^2_{\C}}(1)$ if $j\in\{2,\dots,1+\delta\}$ and $L=\sO_X$. Unfortunately, with these choices, assertion (ii) of Lemma \ref{conditionssurface} would not hold.
\end{rem}

\subsection{Verifying Green's criterion for a  family of double covers}
\label{Greendouble}

To be able to apply Proposition \ref{GreenRgeom} to the family of double covers of $\bP^2_{\R}$, we need a connectedness result going back to Hilbert \cite[p.344]{Hilbert2}, that we state first.

 Let $V:=\C[X_0,\dots X_N]_d$ be the space of degree $d$ homogeneous polynomials $F$ in $N+1$ variables, endowed with its natural real structure, and let $B\subset V$ be the Zariski-open subset parametrizing polynomials $F$ whose zero locus $\{F=0\}\subset\bP^N_{\C}$ is a smooth hypersurface. 
Define $P^{\sm}_d\subset B(\R)$ to be the locus where $F$ is positive semidefinite, and let $P_d^+\subset V(\R)$ be the set of polynomials $F$ such that $F(x)>0$ if $x\in \R^{N+1}\setminus\{0\}$. 

\begin{prop}
\label{openconnected}
The set $P^{\sm}_d$ is open, connected and equal to $B(\R)\cap P_d^+$.
\end{prop}

\begin{proof}
Let $F\in P_d^{\sm}$ and $x\in \R^{N+1}\setminus\{0\}$. If $F(x)>0$ did not hold, then $F(x)=0$, and $x$ would be a smooth point of $\{F=0\}$. Consequently, the differential $dF_x$ would be surjective and $F$ would take negative values near $x$, which is a contradiction. This shows that $P^{\sm}_d=B(\R)\cap P_d^+$.
Since $ P_d^+$ is open and convex \cite[Lemma 4.2]{H17}, it follows that $ P_d^{\sm}$ is open and $P_d^+$ is connected.

The complement of $P_d^{\sm}$ in $P_d^+$ consist of polynomials $F$ such that $\{F=0\}$ is singular and has no real points. It follows that $\{F=0\}$ has at least two singular points (any of them and its distinct complex conjugate). But the set polynomials $F\in V$ such that $\{F=0\}$ has at least two singular points has codimension $\geq 2$, so that $P_d^{\sm}$ is the complement in $P_d^+$ of a semialgebraic set of codimension $\geq 2$. Since the open set $P_d^+\subset V(\R)$ is connected, so is $P_d^{\sm}$.
\end{proof}

For the remainder of this section, we restrict to the case $N=2$. We consider the universal family $\mathcal{S}\subset \bP(1,1,1,\delta)\times B$ with projection $\pi:\mathcal{S}\to B$, parametrizing double covers of $\bP^2$ ramified over a smooth curve of degree $d$: if $b\in B(\C)$ corresponds to the polynomial $F$, one has $\mathcal{S}_b=\{Z^2+F(X_0,X_1,X_2)=0\}\subset\bP(1,1,1,\delta)$. 
The map $\pi$ is a smooth projective morphism of algebraic varieties over $\R$.
We are interested in the fibers of $\pi$ over $P_d^{\sm}\subset B(\R)$.

\begin{lem}
\label{dense}
For a dense set of $b\in P_d^{\sm}$, there exists $\lambda\in H^{1,1}_{\R}(\mathcal{S}_b(\C))(1)^G$ such that the composition:
\begin{equation}
\label{composition2}
T_{B_{\C},b}\longrightarrow H^1(\mathcal{S}_{b,\C},T_{\mathcal{S}_{b,\C}})\stackrel{\lambda}\longrightarrow H^2(\mathcal{S}_{b,\C}, \sO_{\mathcal{S}_{b,\C}})
\end{equation}
of the Kodaira-Spencer map and of the contracted cup-product with~$\lambda$ is surjective.
\end{lem}

\begin{proof}
It follows from the conditions listed  in Lemma \ref{conditionssurface} that the surface $S$ constructed there is isomorphic to a real member of the family $\pi:\mathcal{S}\to B$: there exists $b_0\in B(\R)$ such that $S\simeq\mathcal{S}_{b_0}$. Since $S$ projects to $\Gamma$, $S(\R)=\varnothing$, and we deduce that $b_0\in P_d^{\sm}$. 
By Remark \ref{algcohoclass}, the cohomology class $\lambda_0$ of the line bundle $\mathcal{L}=\sO_S(C)$ associated to the curve constructed in Lemma \ref{conditionssurface} belongs to $H^{1,1}_{\R}(\mathcal{S}_{b_0}(\C))(1)^G$. Moreover, Proposition \ref{surj}, shows that the contracted cup-product $H^1(\mathcal{S}_{b_0,\C},T_{\mathcal{S}_{b_0,\C}})\stackrel{\lambda_0}\longrightarrow H^2(\mathcal{S}_{b_0,\C}, \sO_{\mathcal{S}_{b_0,\C}})$ is surjective.

The deformation theory of the family of smooth double covers of $\bP^2$, carried out in \cite[p.260]{manetti}, shows that the Kodaira-Spencer map $T_{B_{\C},b_0}\to H^1(\mathcal{S}_{b_0,\C},T_{\mathcal{S}_{b_0,\C}})$ is surjective unless $d=6$. Indeed, \cite[(1.3')]{manetti} applied with $Y=\bP^2$ shows that the cokernel of this map embeds into $H^1(\bP^2,T_{\bP^2})\oplus H^1(\bP^2,T_{\bP^2}(-\delta))$, that vanishes when $\delta\neq 3$ and is one-dimensional when $\delta=3$ (as one computes using the Euler exact sequence and Serre duality). 
When $d\neq 6$, this shows at once that the map (\ref{composition2}) is surjective for $(b,\lambda)=(b_0,\lambda_0)$. In the exceptional case $d=6$, the double covers are $K3$ surfaces, and the image of the Kodaira-Spencer map $T_{B_{\C},b_0}\to H^1(\mathcal{S}_{b_0,\C},T_{\mathcal{S}_{b_0,\C}})$ is included in the subspace of polarized infinitesimal deformations, that preserve the line bundle $\sO_{\bP^2}(1)$. This subspace has codimension $1$ (see \cite[Chapter 6, 2.4]{Huybrechts}).
Since the image of $H^0(S,N_{S/X})\to H^1(S, T_{S})$ in (\ref{ksmap}) lands in this subspace, we still deduce from Proposition \ref{surj} the surjectivity of (\ref{composition2}) for $(b,\lambda)=(b_0,\lambda_0)$.

The lemma then follows from the last statement of Proposition \ref{GreenRgeom}, that applies because $P_d^{\sm}\subset B(\R)$ is open and connected by Proposition \ref{openconnected}.
\end{proof}

\subsection{Density} 
\label{parproof}
We are now ready to give the proof of the density statement of Theorem \ref{main}. Recall from \S\ref{sos} that $P_d$ (resp. $Q_d$) is the subset of $\R[x_1,x_2]_{\leq d}$ consisting of polynomials that are positive semidefinite (resp. sums of three squares in $\R(x_1,x_2)$).
\begin{prop}
\label{density}
The set $Q_d$ is dense in $P_d$.
\end{prop}

\begin{proof}
Fix an open subset $U\subset P_d$, and $b\in U\subset \R[X_0,X_1,X_2]_d$ corresponding to the inhomogeneous polynomial $f\in \R[x_1,x_2]_{\leq d}$.
Our goal is to construct $b'\in U$ such that the associated polynomial $f'\in\R[x_1,x_2]_{\leq d}$ is a sum of three squares in $\R(x_1,x_2)$.

Replacing $f$ with $f+t(1+x_1^d+x_2^d)$ for $t\in\R_{>0}$ small enough, we may assume that $b\in P_d^{\sm}$ (see \S\ref{Greendouble}). 
Up to changing $b\in U$ again, Lemma \ref{dense} allows us to suppose that there exists $\lambda\in H^{1,1}_{\R}(\mathcal{S}_b(\C))(1)^G$ such that (\ref{composition2}) is surjective. Consequently, we can choose an open cone $\Omega\subset H^2(\mathcal{S}_b(\C),\R(1))^G$ as in Proposition \ref{GreenRgeom}.

Shrinking $U$, we may assume it is of the form $\Delta(\R)$ for some $\Delta\subset B(\C)$ as in \S\ref{infcrit}.
Denote by $\mathcal{S}(\C)|_{\Delta(\R)}$ the inverse image of $\Delta(\R)$ by $\pi:\mathcal{S}(\C)\to B(\C)$. By a $G$-equivariant version of Ehresmann's theorem \cite[Lemma 4]{Dimca}, it is possible, after shrinking $\Delta$, to ensure that there is a $G$-equivariant diffeomorphism commuting with the projection to $\Delta(\R)$:
 \begin{equation}
\label{Ehresmann}
\mathcal{S}(\C)|_{\Delta(\R)}\stackrel{\sim}\longrightarrow \mathcal{S}_b(\C)\times \Delta(\R).
\end{equation}

Let us now look at the Hochschild-Serre spectral sequence computing the $G$-equivariant cohomology of $\mathcal{S}_b(\C)$:
\begin{equation}
\label{HSss}E_2^{p,q}=H^p(G, H^q(\mathcal{S}_{b}(\C),\Z(1)))\implies H^{p+q}_G(\mathcal{S}_{b}(\C),\Z(1)).
\end{equation}
Since $H_1(\mathcal{S}_b(\C),\Z)=0$ (see \cite[Proposition 6 (i) (ii)]{Dimca}),
 \cite[Lemma 2.3]{vanHamel} shows that the cokernel of the edge morphism $\varepsilon:H^2_G(\mathcal{S}_b(\C),\Z(1))\to H^2(\mathcal{S}_b(\C),\Z(1))^G$ is isomorphic to $\Z/2\Z$. Let $\Lambda$ be the image in $H^2(\mathcal{S}_b(\C),\R(1))^G$ of the complement of the image of $\varepsilon$. Since $\Lambda$ is a translate of a lattice in the real vector space $H^2(\mathcal{S}_b(\C),\R(1))^G$, it meets the open cone $\Omega$. Consequently, we can find a class $\omega\in H^2(\mathcal{S}_b(\C),\Z(1))^G$ not in the image of $\varepsilon$, whose image in $H^2(\mathcal{S}_b(\C),\R(1))^G$, still denoted by $\omega$, belongs to $\Omega$. By construction of $\Omega$, there exists $b'\in \Delta(\R)$ such that the parallel transport $\omega'\in H^2(\mathcal{S}_{b'}(\C),\R(1))^G$  belongs to $H^{1,1}_{\R}(\mathcal{S}_{b'}(\C))(1)^G$.

 For every $k\geq 0$, the two restriction maps $H^k(\mathcal{S}(\C)|_{\Delta(\R)},\Z(1))\to H^k(\mathcal{S}_{b}(\C),\Z(1))$ and $H^k(\mathcal{S}(\C)|_{\Delta(\R)},\Z(1))\to H^k(\mathcal{S}_{b'}(\C),\Z(1))$ are $G$-equivariant isomorphisms by (\ref{Ehresmann}) and contractibility of $\Delta(\R)$. We deduce that the restriction map from the Hochschild-Serre spectral sequence for the $G$-equivariant cohomology
of $\mathcal{S}(\C)|_{\Delta(\R)}$:
$$E_2^{p,q}=H^p(G, H^q(\mathcal{S}(\C)|_{\Delta(\R)},\Z(1)))\implies H^{p+q}_G(\mathcal{S}(\C)|_{\Delta(\R)},\Z(1))$$
to that (\ref{HSss}) for $\mathcal{S}_{b}(\C)$ is an isomorphism in page $2$, hence an isomorphism. The same goes for the restriction map between the Hochschild-Serre spectral sequences for
$\mathcal{S}(\C)|_{\Delta(\R)}$ and $\mathcal{S}_{b'}(\C)$.
We can thus deduce from the corresponding property of $\omega$ that $\omega'\in H^2(\mathcal{S}_{b'}(\C),\R(1))^G$ lifts to a class $\omega'\in H^2(\mathcal{S}_{b'}(\C),\Z(1))^G$ not in the image of the edge map $\varepsilon':H^2_G(\mathcal{S}_{b'}(\C),\Z(1))\to H^2(\mathcal{S}_{b'}(\C),\Z(1))^G$.

Since $\Pic^0(\mathcal{S}_{b',\C})=0$ by \cite[Proposition 6 (i)]{Dimca}, the Lefschetz $(1,1)$ theorem shows that:
$$\Pic(\mathcal{S}_{b',\C})\stackrel{\sim}\longrightarrow \Hdg^2(\mathcal{S}_{b'}(\C),\Z(1)):= H^2(\mathcal{S}_{b'}(\C),\Z(1))\cap H^{1,1}(\mathcal{S}_{b'}(\C)),$$
hence that $\omega'\in\Hdg^2(\mathcal{S}_{b'}(\C),\Z(1))^G$ is the class of a line bundle $\mathcal{L}\in\Pic(\mathcal{S}_{b',\C})^G$. If $\mathcal{L}$ were induced by a real line bundle on $\mathcal{S}_{b'}$, the existence of a cycle class map with value in $G$-equivariant Betti cohomology \cite[\S 1.3]{Krasnov} would show that $\omega'$ lifts to $H^2_G(\mathcal{S}_{b'}(\C),\Z(1))$, a contradiction.

By implication (ii)$\implies$(i) of Lemma \ref{NLsquares}, we deduce that $-1$ is a sum of two squares in $\R(\mathcal{S}_{b'})$. Lemma \ref{23} then shows that
 the polynomial $f'\in\R[x_1,x_2]_{\leq d}$ associated to $b'$ is a sum of three squares in $\R(x_1,x_2)$, which is what we wanted.
\end{proof}

\subsection{The set of sums of three squares}
\label{proof}
We finally complete the:

\begin{proof}[Proof of Theorem \ref{main}]

Let $Q^N_d\subset \R[x_1,x_2]_{\leq d}$ be the set of polynomials $f$ such that there exist polynomials $g,h_1,h_2,h_3\in\R[x_1,x_2]_{\leq N}$ of degree $\leq N$ satisfying:
\begin{equation}
\label{sum3}
fg^2=h_1^2+h_2^2+h_3^2\textrm{ and }g\neq 0.
\end{equation}

  It is an immediate consequence of the Tarski-Seidenberg theorem \cite[Theorem 2.2.1]{BCR} that $Q^N_d$ is a semialgebraic subset of $P_d$.
Let us prove that $Q_d^N$ is a closed subset of $\R[x_1,x_2]_{\leq d}$ by adapting 
 \cite[Theorem 3]{Bergart}. 
Consider the norm $\|h\|:=\sup_{p\in[0,1]^2}h(p)$ on $\R[x_1,x_2]$.
Let $(f_j)_{j\in \mathbb{N}}$ be a sequence of elements of $Q_d^N$ converging to $f\in \R[x_1,x_2]_{\leq d}$. Since the case $f=0$ is trivial, we may assume that the $f_j$ are nonzero. Choose $g_j,h_{j,1},h_{j,2},h_{j,3}\in\R[x_1,x_2]_{\leq N}$ with $g_j\neq 0$ such that:
$$f_jg_j^2=h_{j,1}^2+h_{j,2}^2+h_{j,3}^2.$$
Up to scaling $g_j$ and the $h_{j,i}$, we may assume that $\|f_jg_j^2\|=1$ and as a consequence that $\|h_{j,i}\|\leq 1$ for $1\leq i\leq 3$. Extracting subsequences, we may ensure that the sequences $(g_j)$ and $(h_{j,i})$ converge to polynomials $g,h_i\in\R[x_1,x_2]_{\leq N}$. Taking the limit, we see that $\|fg^2\|=1$ so that $g\neq 0$, and that (\ref{sum3}) holds. Since $Q_d=\cup_{N\in\mathbb{N}} Q_d^N$, this proves the first assertion of Theorem \ref{main}.

If $d\geq 6$, it is a consequence of the Noether-Lefschetz theorem applied as in \cite{CTNL} that $Q_d$ has empty interior. More precisely, every open subset of $\R[x_1,x_2]_{\leq d}$ contains a polynomial whose coefficients are algebraically independent over $\Q$, and such a polynomial cannot be a sum of three squares of rational functions by \cite[Theorem 3.1]{CTNL}. One could also argue as in \cite[Remark 4.3]{CTNL}: $Q_d$ is included in a countable union of proper closed algebraic subvarieties of $\R[x_1,x_2]_{\leq d}$, hence has empty interior by a Baire category argument.

Finally, we have proven the last assertion of Theorem \ref{main} in Proposition \ref{density}.
\end{proof}

\section{Surfaces whose function field has level $2$}\label{sec4}

The proof of Theorem \ref{main2} is analogous to that of Proposition \ref{density}. 
We fix an even integer $d\geq 2$.

Consider $X:=\bP^3_{\C}$, define $B\subset \bP(H^0(X,\sO_X(d)))$ to be the subset parametrizing equations $F$ defining smooth surfaces $S\subset X$, and $\pi:\mathcal{S}\to B$ to be the universal surface. Endow $X$, $B$ and $\mathcal{S}$ with their natural real structures.  Recall from \S\ref{introlevel} that $\Theta_d\subset B(\R)$ is the set of equations $F$ whose associated surfaces $S=\{F=0\}$ have no real points. Since $\R^4\setminus\{0\}$ is connected, $\Theta_d$ consists of equations, well-defined up to a scalar, that are either positive or negative on $\R^4\setminus\{0\}$
and Proposition \ref{openconnected} implies that $\Theta_d \subset B(\R)$ is open and connected.

Set $n:=d$ if $d\equiv 0 \pmod 4$ and $n:=d-2$ if $d\equiv 2 \pmod 4$. Define $H_j:=\sO_{X}(1)$ for $1\leq j\leq n$, $L:=\sO_X$ if $d\equiv 0 \pmod 4$ and $L:=\sO_X(2)$ if $d\equiv 2 \pmod 4$. Note that Hypotheses \ref{hypo} are satisfied. When $d\equiv 2 \pmod 4$, those are exactly the original choices of Ciliberto and Lopez \cite{CL}.

Consider $(n+1)\times n$ matrices $M=(M_{i,j})_{1\leq i\leq n+1, 1\leq j\leq n}$ with $M_{i,j}\in H^0(X,H_j)$.
To such a matrix $M$, we associate the variety $C\subset X$ defined by the vanishing of all the maximal minors of $M$. We also consider a section $\tau\in H^0(X,\sO(d))$ vanishing on $C$ with zero locus $S\subset X$.
The following is an analogue of Lemma \ref{conditionssurface}:

\begin{lem}
\label{conditionssurface2}
It is possible to find $M$ and $\tau$ such that:
\begin{enumerate}[(i)]
\item $C$ is a smooth curve and $S$ is a smooth surface,
\item the subvarieties $C$ and $S$ of $X$ are defined over $\R$,
\item $S(\R)=\varnothing$.
\end{enumerate}
\end{lem}

\begin{proof}
Since $n\equiv 0 \pmod 4$, one may consider a particular choice $M^0$ of $M$, that is defined over $\R$, and whose $n\times n$ submatrix $(M^0_{i,j})_{1\leq i,j\leq n}$ is diagonal by blocks with blocks:
$$\begin{pmatrix}
X_0 & -X_1 &X_2 &  X_3\\
X_1 & X_0 &X_3& -X_2\\
X_2 & X_3 & -X_0 &X_1\\
X_3 & -X_2 &-X_1&-X_0 
\end{pmatrix}.$$
Computing that the determinant of every such block is $\sum_{0\leq i,j\leq 3}X_i^2X_j^2$, one sees that $\eta^0:=\det((M^0_{i,j})_{1\leq i,j\leq n})\in H^0(X,\sO_X(n))$ does not vanish on $X(\R)$. Let $M$ be a general small real perturbation of $M^0$: the property that $\eta:=\det((M_{i,j})_{1\leq i,j\leq n})$ does not vanish on $X(\R)$ persists.  Choose $\tau$ to be a general small real perturbation of $\eta$ if $d\equiv 0 \pmod 4$ (resp. of $(X_0^2+X_1^2+X_2^2+X_3^2)\cdot\eta$ if $d\equiv 2 \pmod 4$). By  Lemma \ref{smooth}, the curve $C$ and the surface $S$ are smooth, and the pair $(M,\tau)$ satisfies the required conditions.
\end{proof}

\begin{rem}
When $d\equiv 2\pmod 4$, it is not possible to run our strategy with $n=d$ and $L=\sO_X$ as when $d\equiv 0\pmod 4$. Indeed, the degree of the determinantal curve $C$, that is equal to $\frac{n(n+1)}{2}$ by \cite[Proposition 12 (a)]{HarrisTu},
 would be odd. Consequently, $C$ would have a real point and assertion (iii) of Lemma \ref{conditionssurface2} could not hold.
\end{rem}

We deduce an analogue of Lemma \ref{dense}:

\begin{lem}
\label{dense2}
For a dense set of $b\in \Theta_d$, there exists $\lambda\in H^{1,1}_{\R}(\mathcal{S}_b(\C))(1)^G$ such that the composition of the Kodaira-Spencer map and of the contracted cup-product with~$\lambda$:
\begin{equation}
\label{composition3}
T_{B_\C,b}\longrightarrow H^1(\mathcal{S}_{\C,b},T_{\mathcal{S}_{\C,b}})\stackrel{\lambda}\longrightarrow H^2(\mathcal{S}_{\C,b}, \sO_{\mathcal{S}_{\C,b}})
\end{equation}
is surjective.
\end{lem}

\begin{proof}
The surface $S$ constructed in Lemma \ref{conditionssurface2} is isomorphic to a particular member $\mathcal{S}_{b_0}$ of the family $\pi:\mathcal{S}\to B$, corresponding to a point $b_0\in \Theta_d$. 
By Remark \ref{algcohoclass}, the cohomology class $\lambda_0$  of the real curve $C\subset S$ constructed in Lemma \ref{conditionssurface2} belongs to $H^{1,1}_{\R}(\mathcal{S}_{b_0}(\C))(1)^G$. By Proposition \ref{surj}, the map (\ref{composition3}) for $b=b_0$ and $\lambda=\lambda_0$ is surjective.
The lemma then follows from the last part of Proposition \ref{GreenRgeom}, that applies because $\Theta_d\subset B(\R)$ is open and connected.
\end{proof}

We can now give the:

\begin{proof}[Proof of Theorem \ref{main2}]
Fix an open subset $U\subset \Theta_d$, and $b\in U$ according to Lemma \ref{dense2}. Running the proof of Proposition \ref{density} (replacing $P_d^{\sm}$ by $\Theta_d$) shows that there exists $b'\in U$ such that $\Pic(\mathcal{S}_{b'})\to\Pic(\mathcal{S}_{b',\C})^G$ is not surjective.
Implication (ii)$\implies$(i) of Lemma \ref{NLsquares} concludes.
\end{proof}

\begin{acknowledgements}
I would like to thank Olivier Wittenberg for numerous discussions on related topics, as well as suggestions to improve this paper.
\end{acknowledgements}

\bibliographystyle{plain}
\bibliography{biblio}

\end{document}